\documentclass[review,onefignum,onetabnum]{siamart190516}

\usepackage{lipsum}
\usepackage{amsfonts}
\usepackage{graphicx}
\usepackage{epstopdf}
\usepackage{algorithmic}
\ifpdf
  \DeclareGraphicsExtensions{.eps,.pdf,.png,.jpg}
\else
  \DeclareGraphicsExtensions{.eps}
\fi

\newsiamremark{hypothesis}{Hypothesis}
\crefname{hypothesis}{Hypothesis}{Hypotheses}
\newsiamthm{claim}{Claim}

\title{Convergence Analysis of Block Newton Methods \\ for 1D Shallow Neural Network Approximation\thanks{This work was performed under the auspices of the U.S. Department of Energy by Lawrence Livermore National Laboratory under Contract DE-AC52-07NA27344. This work was supported by the NNSA Advanced Simulation and Computing (ASC) Program.  LLNL-JRNL-2015820.}}

\author{Zhiqiang Cai\thanks{Department of Mathematics, Purdue University, West Lafayette, IN
    (\email{caiz@purdue.edu},  \email{adoktoro@purdue.edu}, \email{herre125@purdue.edu}).}
\and Anastassia Doktorova\footnotemark[2]
\and Robert D. Falgout\thanks{Lawrence Livermore National Laboratory, Livermore, CA (\email{rfalgout@llnl.gov})}
\and César Herrera\footnotemark[2]
}

\usepackage{amsopn}

\usepackage{amsmath,rotating}

\usepackage[utf8]{inputenc}
\usepackage{bm}
\usepackage{mathtools}
\usepackage{indentfirst}
\usepackage{tcolorbox}
\usepackage{color}
\usepackage[export]{adjustbox}
\usepackage{scalerel}
\usepackage{tikz}
\usepackage{geometry}
\usepackage{multicol}
\usetikzlibrary{decorations.pathreplacing}
\usepackage{algorithmic}
\Crefname{ALC@unique}{Line}{Lines}

\makeatletter
\let\orig@nomath\@nomath
\let\orig@scriptsize\scriptsize
\renewcommand{\scriptsize}{%
  \ifmmode
    \let\@nomath\@gobble 
    \orig@scriptsize
    \let\@nomath\orig@nomath 
  \else
    \orig@scriptsize
  \fi}
\makeatother

\newcommand{\N}{\mathbb{N}}
\usepackage{graphicx}
\usepackage{diagbox}
\usepackage{pifont}
\newcommand{\vertiii}[1]{{\left\vert\kern-0.25ex\left\vert\kern-0.25ex\left\vert #1 
    \right\vert\kern-0.25ex\right\vert\kern-0.25ex\right\vert}}

\newcommand{\btheta}{\mbox{\boldmath${\theta}$}}

\usepackage{tabularray}
\usepackage{caption}
\usepackage{subcaption}
\usepackage{soul}

\usepackage{enumitem}
\setlist[itemize]{left=16pt} 

\def\bb{{\bf b}}
\def\bB{{\bf B}}
\def\bc{{\bf c}}
\def\bd{{\bf d}}
\def\bD{{\bf D}}

\def\bg{{\bf g}}

\def\bv{{\bf v}}

\def\bH{{\bf H}}
\def\bJ{{\bf J}}
\def\balpha{\boldsymbol{\alpha}}
\def\bbeta{\boldsymbol{\beta}}

\def\cM{{\cal M}}

\DeclareMathOperator*{\argmin}{arg\,min}

\begin{document}

\maketitle
\begin{abstract}
This paper analyzes local convergence of the block Newton (BN) method introduced in \cite{cai2024fast, cai2024fast2} for one-dimensional shallow neural network approximation to functions and diffusion-reaction problems. The BN method  consists of the $2\times 2$ block nonlinear Gauss-Seidel, linear Gauss-Seidel, or Jacobi method for outer iteration and the Newton method for inner iteration. The blocks are corresponding to the linear and the nonlinear parameters. Under some reasonable assumptions, we establish local convergence of the BN methods as well as the reduced BN (rBN) method for one-dimensional diffusion–reaction problems and least-squares function approximation. Unlike common optimization methods, the rBN allows for the reduction of the number of parameters during the optimization process when some neurons contribute little to the approximation or are at nearly optimal locations.

\end{abstract}

\begin{keywords}
    Neural network, Elliptic problems, Least-Squares approximation, Newton's method, Local convergence analysis
\end{keywords}

\section{Introduction}

One dimensional ReLU shallow neural network (NN) with $n$ neurons generates a set of continuous piecewise linear functions. 
Specifically, the set with restriction of the biases in the interval $I = (0,1)$ is given by
\begin{equation*}\label{ShallowNN}
    {\cal M}_n(I) = \left\{ \alpha + \sum_{i=0}^{n} c_i \, \sigma(x - b_i) \; : \; \alpha \in \mathbb{R},\; c_i \in \mathbb{R},\; 0 = b_0 < b_1 < \cdots < b_n < b_{n+1} = 1 \right\},
\end{equation*}
where $\sigma(t) = \max\{0, t\}$ is the ReLU activation function. Denote by $\bc = (c_0, \dots, c_n)^T \in \mathbb{R}^{n+1}$ and $\bb = (b_1, \dots, b_n)^T \in \mathbb{R}^n$ the respective linear parameters and nonlinear parameters, and denote by $\btheta = (c_0, c_1, \dots, c_n, b_1, \dots, b_n)^T \in \mathbb{R}^{2n+1}$ all parameters. Notice that the weights of all hidden layer is chosen to be one by normalization (see \cite{Liu_Cai_Chen_2022}). Each function in ${\cal M}_n(I)$ 
\[
v(x; \bc,\bb) = \alpha + \sum_{i=0}^{n} c_i \, \sigma(x - b_i)
\]
depends on the parameters $\bc$ and $\bb$ and is a piecewise linear function of $x\in [0,1]$ with respect to the partition by the breaking (mesh) points: $0 = b_0 < b_1 < \cdots < b_n < b_{n+1} = 1$. 

It is known (see, e.g., \cite{Liu_Cai_Chen_2022}) that the set ${\cal M}_n(I)$ is equivalent to the free-knot splines (FKS) (see \cite{Schumaker}), where FKS utilizes local hat basis functions. FKS can substantially enhance the approximation order and reduce the number of degrees of freedom for non-smooth functions (see \cite{Burchard74} and the discussion in \cite{cai2024fast}). As an example of this, the order of the best approximation to $f(x)=x^{\alpha}$ ($0 < \alpha < 1$) on $I$ is merely $\alpha<1$ (i.e., ${\cal O}(n^{-\alpha})$) when using finite elements on a fixed uniform mesh, whereas for FKS, the order becomes one (i.e., ${\cal O}(n^{-1})$) no matter how small the exponent $\alpha>0$ is (see, e.g., \cite{daubechies, Schumaker}). This is a huge improvement in approximation.

Despite the remarkable approximation capability of FKS for non-smooth functions, there are two essential difficulties that have led numerical analysts moving away from FKS: (1)~no successful extension of FKS to two or higher dimensions has been achieved, and (2)~determining the optimal knot locations (the nonlinear parameters $\mathbf{b}$) results in a high-dimensional, non-convex optimization problem, that is computationally expensive and hence dismisses its benefit in approximation. To the best of our knowledge, there are no available efficient optimization schemes that enable the competitiveness of FKS. While the first issue may be addressed by employing neural networks due to their global supported basis functions, the second issue still remains a major challenge.


To address this challenge, a major advance on fast iterative solver was recently made in \cite{cai2024fast, cai2024fast2} for solving non-convex optimization problems arising from shallow ReLU NN approximation to a given function or solutions of elliptic differential equations in one dimension. Specifically, a well-designed damped block Newton (dBN) method was developed. First, the dBN adopts a classical outer-inner iterative strategy (see, e.g., \cite{Ainsworth2020PlateauPI, Ainsworth2022, pmlr-v107-cyr20a, Golub1973, separable3, park2022neuron, separable2}), alternating between updates of the linear and nonlinear parameters. Second, to solve the resulting dense, ill-conditioned linear systems due to the global basis functions of NNs, the dBN uses the fact that the exact inversion of those matrices can be represented in terms of products of sparse matrices (see \cite{cai2024fast2}). Third, the dBN deals directly with singularities of the Hessian for the nonlinear parameters by removing neurons whose linear parameters are small or whose nonlinear parameters have reached nearly optimal locations. As a result, the computational cost per iteration of the dBN is ${\cal O}(n)$, and numerical experiments show that this method is capable of moving mesh points effectively and efficiently. Beyond strong one dimensional results, the methodology of the dBN is conceptually promising for higher dimensions, as it demonstrates how to design iterative solvers that exploit the problem structure and the approximation and geometric meanings of the NN parameters.

The purpose of this paper is to provide a theoretical guarantee on why this sophisticated dBN moves the mesh points efficiently. This will be done by analyzing local convergence of block Newton (BN) methods. 
The BN method consists of the $2\times 2$ block nonlinear Gauss-Seidel, linear Gauss-Seidel, or Jacobi method for outer iteration and the Newton method for inner iteration. The blocks are corresponding to the linear and the nonlinear parameters. 
By following the machinery in \cite{ortega1970iterative} on local convergence of the componentwise Gauss-Seidel method, we first develop a local convergence theory for the block Newton methods applicable to shallow ReLU NNs in both one and multiple dimensions, provided that the Hessian matrix at a critical point is symmetric positive definite (SPD) and that the $2\times 2$ block nonlinear Gauss-Seidel, linear Gauss-Seidel, or Jacobi matrix is invertible. 

By expressing the BN method as a fixed-point iteration, local convergence of the BN method is established by showing that a norm of the corresponding Jacobian at the critical point is strictly less than one. Note that derivation of the Jacobian matrix is non-trivial (see \cref{BN_lemma1}). 
To guarantee feasibility of each Newton step, the BN method is modified to the reduced BN (rBN) method that allows a reduction in the number of parameters during the optimization process. Local convergence of the rBN is justified by showing that some nonlinear parameters are at nearly optimal locations.

The paper is structured as follows. \Cref{Section_BN_method} introduces BN methods with three different outer iteration methods. Local convergence analysis of these methods are presented in \Cref{Section_convergence}. Sufficient conditions on SPD of the Hessian for one-dimensional problems are derived in \Cref{Apps}. \Cref{implementatioSection} discusses the reduced Block Newton method and some conclusions and remarks are presented in \Cref{sec:conc}.

\section{Block Newton Methods}\label{Section_BN_method} Let $\btheta =\begin{pmatrix} \bc \\ \bb \end{pmatrix}$, where $\bc \in \mathbb{R}^{n+1}$ and $\bb \in \mathbb{R}^{n}$.
Given an open set ${\cal D} \subseteq \mathbb{R}^{2n+1}$ and a twice continuously differentiable function $F  = F(\btheta) = F(\bc, \bb):  
{\cal D} \rightarrow \mathbb{R}$, we aim to find a minimizer $\btheta^* = \begin{pmatrix} \bc^{*} \\ \bb^{*} \end{pmatrix} \in {\cal D}$ such that 
 \begin{equation*}\label{min_general}
     F(\btheta^{*})  = \min_{\bm{\scriptstyle\theta} \in {\cal D}}F(\btheta).
 \end{equation*}
Optimality conditions imply that $\btheta^{*}$ satisfies the system of nonlinear algebraic equations
\begin{equation}\label{opt_conditions}
 \nabla_{\scriptsize\btheta} F\left(\btheta \right) = \begin{pmatrix}  \nabla_{\bc} F\left(\btheta\right)
 \\[2mm] \nabla_{\bb} F\left(\btheta\right)\end{pmatrix} 
 = \begin{pmatrix}
 {\bf 0} \\[2mm] {\bf 0}\end{pmatrix}
\end{equation}
where $\nabla_{\bc}$ and $\nabla_{\bb}$ denote the gradients of $F(\btheta)$ with respect to $\bc$ and $\bb$, respectively. The Hessian matrix is given by 
\begin{equation*}\label{fullHessian}
\nabla^2_{\scriptsize\btheta}F(\btheta) =  \begin{pmatrix} \nabla^2_{\bc\bc}F(\btheta) & \nabla^2_{\bc \bb}F(\btheta)\\[2mm] \nabla^2_{\bb \bc}F(\btheta) & \nabla^2_{\bb\bb}F(\btheta)  \end{pmatrix}
=\begin{pmatrix} {\cal H}_{11}(\btheta) & {\cal H}_{12}(\btheta)\\[2mm] {\cal H}_{21}(\btheta) & {\cal H}_{22}(\btheta)  \end{pmatrix}
     \in \mathbb{R}^{(2n+1) \times (2n+1)},
\end{equation*}
where ${\cal H}_{ij}(\btheta)$ for $i,j=1,2$ are given by
\begin{equation*}\label{defOperatos}
    {\cal H}_{11}(\btheta) = \nabla^2_{\bc\bc}F(\btheta), \quad
{\cal H}_{12}(\btheta) = \nabla^2_{\bc \bb}F(\btheta), \quad
{\cal H}_{21}(\btheta) = \nabla^2_{\bb \bc}F(\btheta), \quad
{\cal H}_{22}(\btheta) = \nabla^2_{\bb\bb}F(\btheta).
\end{equation*}


The nonlinear system in \cref{opt_conditions} can be solved using Newton's method, though this may be computationally expensive. To reduce cost, we can use a block Newton (BN) method by performing an outer-inner iteration that alternates between the variables $\bc$ and $\bb$. For the outer iteration, one may use a Gauss--Seidel or Jacobi scheme, and apply a Newton iteration to each block during the inner solve.

More specifically, let $\left(\bc^{(k)}, \bb^{(k)}\right)$ denote the current iterate. Then the block nonlinear Gauss-Seidel (see, e.g., \cite{ortega1970iterative}) method, as the outer iteration, computes the new iterate $\left(\bc^{(k+1)}, \bb^{(k+1)}\right)$ as follows:
\begin{enumerate}
\item \textit{Update the variable $\bc^{(k+1)}$}:
\begin{equation*}
\bc^{(k+1)} = \bc^{(k)} - \left[{\cal H}_{11}\left(\bc^{(k)}, \bb^{(k)}\right)\right]^{-1} \nabla_{\bc} F\left(\bc^{(k)}, \bb^{(k)}\right).
\end{equation*}
\item \textit{Update the variable $\bb^{(k+1)}$}:
\begin{equation*}
\bb^{(k+1)} = \bb^{(k)} - \left[ {\cal H}_{22}\left(\bc^{(k+1)}, \bb^{(k)}\right)\right]^{-1} \nabla_{\bb} F\left(\bc^{(k+1)}, \bb^{(k)}\right).
\end{equation*}
\end{enumerate}
In other words, the new iterate $\left(\bc^{(k+1)},\bb^{(k+1)}\right)$ is the solution of the block diagonal system of nonlinear algebraic equations
\begin{equation}\label{GS}
 \begin{pmatrix} {\cal H}_{11}\left(\bc^{(k)}, \bb^{(k)}\right) & {\bf 0} \\[2mm] 
 {\bf 0} & {\cal H}_{22}\left(\bc^{(k+1)}, \bb^{(k)}\right)
 \end{pmatrix}  
 \begin{pmatrix}
     \bc^{(k+1)} - \bc^{(k)} \\[2mm] \bb^{(k+1)} - \bb^{(k)}
 \end{pmatrix}
 = -\begin{pmatrix}
     \nabla_{\bc} F\left(\bc^{(k)}, \bb^{(k)}\right) \\[2mm] \nabla_{\bb} F\left(\bc^{(k+1)}, \bb^{(k)}\right)
 \end{pmatrix},
\end{equation}
which can be solved sequentially by computing two systems of linear algebraic equations.

This nonlinear Gauss-Seidel method differs from the classical linear Gauss--Seidel method, where the new iterate is obtained by solving the block lower-triangular system of linear equations
\begin{equation}\label{GS2}
 \begin{pmatrix} {\cal H}_{11}\left(\bc^{(k)}, \bb^{(k)}\right) & {\bf 0} \\[2mm] 
 {\cal H}_{21}(\bc^{(k)}, \bb^{(k)}) & {\cal H}_{22}\left(\bc^{(k)}, \bb^{(k)}\right)
 \end{pmatrix}  
 \begin{pmatrix}
     \bc^{(k+1)} - \bc^{(k)} \\[2mm] \bb^{(k+1)} - \bb^{(k)}
 \end{pmatrix}
 = -\begin{pmatrix}
     \nabla_{\bc} F\left(\bc^{(k)}, \bb^{(k)}\right) \\[2mm] \nabla_{\bb} F\left(\bc^{(k)}, \bb^{(k)}\right)
 \end{pmatrix}.
\end{equation}
On the other hand, in the Jacobi method, the new iterate $\bigl(\bc^{(k+1)}, \bb^{(k+1)}\bigr)$ solves the block diagonal system of linear algebraic equations

\begin{equation}\label{Jacobi}
 \begin{pmatrix} {\cal H}_{11}\left(\bc^{(k)}, \bb^{(k)}\right) & {\bf 0} \\[2mm] 
 {\bf 0} & {\cal H}_{22}\left(\bc^{(k)}, \bb^{(k)}\right)
 \end{pmatrix}  
 \begin{pmatrix}
     \bc^{(k+1)} - \bc^{(k)} \\[2mm] \bb^{(k+1)} - \bb^{(k)}
 \end{pmatrix}
 = -\begin{pmatrix}
     \nabla_{\bc} F\left(\bc^{(k)}, \bb^{(k)}\right) \\[2mm] \nabla_{\bb} F\left(\bc^{(k)}, \bb^{(k)}\right)
 \end{pmatrix}.
\end{equation}



    

\section{Convergence Analysis}\label{Section_convergence}  For brevity, we refer to the schemes defined in \cref{GS}, \cref{GS2}, and \cref{Jacobi} as NL-GS, L-GS, and JB, respectively. This section presents analytic tools for deriving local convergence conditions for the BN methods introduced in the previous section. To this end, we introduce the following assumption:
\begin{itemize}
    \item \noindent\textbf{Invertibility Assumption:} there exists an open set $\mathcal{O} \subseteq \mathcal{D}$ such that $\mathcal{H}_{11}(\boldsymbol{\theta})$ is invertible; moreover, $\mathcal{H}_{22}(G_1(\boldsymbol{\theta}), \mathbf{b})$ is invertible for NL-GS and $\mathcal{H}_{22}(\boldsymbol{\theta})$ is invertible for L-GS and JB.
\end{itemize}

Define the mapping 
\[
G(\btheta) = \begin{pmatrix} G_1(\btheta)\\[1mm] G_2(\btheta) \end{pmatrix}, 
\]
where $G_1: {\cal D} \to \mathbb{R}^{n+1}$ and $G_2: {\cal D} \to \mathbb{R}^n$ are respectively given by
\begin{equation}\label{G1BN}
    G_1(\btheta) =
\bc - {\cal H}_{11}^{-1}(\btheta)\,\nabla_{\bc} F(\btheta),
\end{equation}
and by
\begin{equation}\label{GBN}
\begin{aligned}
G_2(\btheta) &=
\begin{cases}
\bb - {\cal H}_{22}^{-1}\!\left(G_1(\btheta), \bb\right)\nabla_{\bb}F\!\left(G_1(\btheta), \bb\right), & \text{for NL-GS}, \\[2mm]
\bb - {\cal H}_{22}^{-1}(\btheta)\,\biggl(\nabla_{\bb}F(\btheta) + {\cal H}_{21}(\btheta)\,\bigl(G_1(\btheta) - \bc\bigr)\biggr), & \text{for L-GS}, \\[2mm]
\bb - {\cal H}_{22}^{-1}(\btheta)\,\nabla_{\bb}F(\btheta), & \text{for JB}.
\end{cases}
\end{aligned}
\end{equation}
Then the BN methods can be expresed as the fixed-point iteration of $G(\btheta)$, i.e., 
\begin{equation}\label{fixedPoint}
    \btheta^{k+1} = G(\btheta^{k}) \quad\mbox{for $k \in \N=\{ 0, 1, 2,\dots\}$}.
\end{equation}

Denote by $\bJ_{G}(\btheta) \in \mathbb{R}^{(2n+1) \times (2n+1)}$ the Jacobian matrix of $G$ at $\btheta$.
A sufficient condition for local convergence of the fixed-point iteration in \cref{fixedPoint} was given in Theorem 10.1.3 of \cite{ortega1970iterative}. For the convenience of readers, we state it below and provide its proof. 

\begin{theorem}[Ostroswki]\label{thmOS} 
Suppose that $G : {\cal O} \rightarrow \mathbb{R}^{2n+1}$ has a fixed point $\btheta^{*} \in {\cal O}$ and that the mapping $G$ is differentiable at $\btheta^{*}$. Denote by $\|\cdot\|$ a norm in $\mathbb{R}^{2n+1}$, if $\|{\bf J}_{G}(\btheta^{*})\| = \sigma < 1$, then the fixed-point iteration in \cref{fixedPoint} converges locally to $\btheta^{*}$. 
\end{theorem}

\begin{proof}
For any $0<\epsilon < 1-\sigma$, the assumption on the differentiability of $G$ at $\btheta^{*}$ implies that there exists a $\delta > 0$ neighbourhood centered at $\btheta^{*}$, ${\cal B}(\btheta^{*}; \delta) = \{\btheta \in \mathbb{R}^{2n+1} : \|\btheta - \btheta^{*}\| < \delta\} \subset {\cal O}$, such that 
\begin{equation}\label{3.4a}
\left\|G(\btheta) - \left[G(\btheta^{*}) + \bJ_G(\btheta^{*})(\btheta - \btheta^{*})\right]\right\|
        \leq \epsilon\|\btheta - \btheta^{*}\| \quad \text{for all } \btheta \in {\cal B}(\btheta^{*}; \delta).    
\end{equation}
For any $k\in \N$, the assumption that $\btheta^{*}=G(\btheta^{*})$ and \cref{fixedPoint} give
\[
\btheta^{k+1}- \btheta^{*}= G(\btheta^{k})- G(\btheta^{*})= \biggl(G(\btheta^k) - \left[G(\btheta^{*}) + \bJ_G(\btheta^{*})(\btheta^k - \btheta^{*})\right]\biggr) + \bJ_G(\btheta^{*})(\btheta^k - \btheta^{*}),
\]
which, together with the triangle inequality, the assumption that $\btheta^k\in {\cal B}(\btheta^{*}; \delta)$, and \cref{3.4a}, implies
\[
\|\btheta^{k+1} - \btheta^{*}\| \leq (\epsilon + \sigma)\|\btheta^k - \btheta^{*}\| < \|\btheta^k - \btheta^{*}\|.
\]
If the initial $\btheta^0$ belongs to ${\cal B}(\btheta^{*}; \delta)$, then the second inequality implies that $\btheta^k\in {\cal B}(\btheta^{*}; \delta)$ by induction. Hence, we have 
$\|\btheta^{k+1} - \btheta^{*}\| \leq (\epsilon + \sigma)^k\|\btheta^0 - \btheta^{*}\|$, which proves the theorem.
\end{proof}

Next, let  
\[
B(\btheta) = \begin{pmatrix} {\cal H}_{11}(\btheta) & {\bf 0}\\[1mm]{\cal H}_{21}(\btheta) &{\cal H}_{22}(\btheta) \end{pmatrix} \,\mbox{ and }\,  B_1(\btheta)= \begin{pmatrix} {\cal H}_{11}(\btheta) & {\bf 0}\\[1mm]{\bf 0}  &{\cal H}_{22}(G_1(\btheta),\bb)  \end{pmatrix}. 
\]
Assume that $\btheta^*$ is a minimizer of $F(\btheta)$, then  
Taylor's expansion of $\nabla_{\scriptsize\btheta}F(\btheta)$ at $\btheta^*$ gives
\begin{equation}\label{Taylor1}
    \nabla_{\scriptsize\btheta}F(\btheta) =  
    \nabla^2_{\scriptsize\btheta}F(\btheta^*) (\btheta-\btheta^*) + R(\btheta;\btheta^*),
\end{equation}
where $R(\btheta;\btheta^*)$ is the remainder satisfying $\lim\limits_{{\scriptsize\btheta} \to {\scriptsize\btheta^{*}}} \|R(\btheta;\btheta^*)\| \big/ \|\btheta - \btheta^*\| = 0$. Similarly, expanding $\nabla_{\bb}F(G_1(\btheta), \bb)$ about $\btheta$ and using \cref{G1BN}, we have
\[ 
    \nabla_{\bb}F(G_1(\btheta), \bb) 
    =\nabla_{\bb}F(\btheta) -{\cal H}_{21}(\btheta){\cal H}_{11}^{-1}(\btheta) \nabla_{\bc}F(\btheta) + \tilde{R}(\btheta),
\] 
where the remainder $\tilde{R}(\btheta)$ satisfies $\lim\limits_{{\scriptsize\btheta} \to {\scriptsize\btheta^{*}}} \dfrac{\|\tilde{R}(\btheta)\|}{ \|\btheta - \btheta^*\|} = 0$. Let $B_2(\btheta) = \begin{pmatrix}I & {\bf 0} \\[2mm]   - {\cal H}_{21}(\btheta){\cal H}_{11}^{-1}(\btheta) & I         \end{pmatrix}$, then combining with \cref{Taylor1} yields
\begin{equation}\label{Taylor2} 
\begin{pmatrix} \nabla_{\bc}F(\btheta) \\[2mm]            \nabla_{\bb}F(G_1(\btheta), \bb)          \end{pmatrix} = B_2(\btheta) \nabla_{\scriptsize{\btheta}}F(\btheta) + \begin{pmatrix} {\bf 0} \\[2mm]  \tilde{R}(\btheta) \end{pmatrix}  
= B_2(\btheta) \nabla^2_{\scriptsize{\btheta}}F(\btheta^*) (\btheta-\btheta^*) + \hat{R}(\btheta;\btheta^*),
\end{equation}
where $\hat{R}(\btheta;\btheta^*)= B_2(\btheta) {R}(\btheta;\btheta^*) + \begin{pmatrix} {\bf 0} \\[2mm]  \tilde{R}(\btheta) \end{pmatrix}$.

The next lemma provides a formula for $\bJ_G(\btheta^{*})$ when $\nabla^2_{\scriptsize{\btheta}}F(\btheta^{*})$ is symmetric and positive definite (SPD) for Gauss--Seidel schemes. For the componentwise Gauss--Seidel method, this result was proved in Theorem 10.3.3 of \cite{ortega1970iterative}. Here, we use block Gauss--Seidel methods.

\begin{lemma}\label{BN_lemma1}
    Let $G: {\cal O} \rightarrow \mathbb{R}^{2n+1}$ be the mapping defined in \cref{GBN} for the L-GS or NL-GS, and let $\btheta^{*} \in {\cal O}$ be a fixed point of $G$. Assume that $\nabla_{\scriptsize\btheta}^2F(\btheta^{*})$ is SPD. Then
    \begin{equation}\label{jacG}
        \bJ_{G}(\btheta^{*}) = I_{2n+1 } - B^{-1}(\btheta^{*})\nabla_{\scriptsize\btheta}^2F(\btheta^{*}),
    \end{equation}
    where $I_{2n+1}$ is the order-$(2n+1)$ identity matrix.
\end{lemma}

\begin{proof}
By definition of differentiability, to show the validity of \cref{jacG}, it suffices to show that if $\bJ_{G}(\btheta^{*})$ is given in \cref{jacG}, then
\begin{equation}\label{lim1}  
\lim_{\bm{\scriptstyle\theta} \to \bm{\scriptstyle\theta}^{*}} 
        \frac{\left\|G(\btheta) - \left[G(\btheta^*) + \bJ_{G}(\btheta^{*})(\btheta - \btheta^*)\right]\right\| }{\|\btheta - \btheta^*\|} = 0.
\end{equation}
To this end, for $\bJ_{G}(\btheta^{*})$ given in \cref{jacG}, we have  
\[ 
B(\btheta^{*}) \bJ_{G}(\btheta^{*}) 
= \begin{pmatrix} {\bf 0} & -{\cal H}_{12}(\btheta^{*})\\ {\bf 0} & {\bf 0} \end{pmatrix} \quad\mbox{and}\quad B(\btheta^{*}) - B(\btheta^{*}) \bJ_{G}(\btheta^{*}) = \nabla_{\scriptsize\btheta}^2F(\btheta^*),
\] 
which, together with the assumption that $\btheta^*$ is a fixed-point of $G(\btheta)$, implies
\begin{equation*}\label{3.9a}
  {\bf a}(\btheta,\btheta^{*})\equiv B(\btheta^{*}) \bigg(G(\btheta) -\big[G(\btheta^*) + \bJ_{G}(\btheta^{*})(\btheta - \btheta^*)\big]\bigg) =  B(\btheta^{*}) \Big(G(\btheta) -\btheta \Big) + \nabla_{\scriptsize\btheta}^2F(\btheta^*) (\btheta -\btheta^*). 
\end{equation*}
For the L-GS, \cref{Taylor1} gives
\[ 
    G(\btheta) -\btheta = -B^{-1}(\btheta) \nabla_{\scriptsize\btheta}F(\btheta) = - B^{-1}(\btheta)\nabla_{\scriptsize\btheta}^2F(\btheta^*) (\btheta -\btheta^*) -B^{-1}(\btheta)R(\btheta;\btheta^*),
\] 
which implies
\begin{equation}\label{L-GS}
 {\bf a}(\btheta,\btheta^{*}) = \Bigl( B(\btheta) - B(\btheta^{*})\Bigr)B^{-1}(\btheta) \nabla_{\scriptsize\btheta}^2F(\btheta^*) (\btheta -\btheta^*) -B(\btheta^{*}) B^{-1}(\btheta)R(\btheta;\btheta^*).    
\end{equation}

For the NL-GS, (\ref{Taylor2}) leads to
\[ 
    G(\btheta) -\btheta 
    = - B^{-1}_1(\btheta) B_2(\btheta)\nabla_{\scriptsize\btheta}^2F(\btheta^*) (\btheta -\btheta^*) -B^{-1}_1(\btheta) \hat{R}(\btheta;\btheta^*), 
\] 
which, together with $B_3(\btheta)\equiv B_2^{-1}(\btheta)B_1(\btheta) = \begin{pmatrix}{\cal H}_{11}(\btheta) & {\bf 0} \\[2mm]   {\cal H}_{21}(\btheta) & {\cal H}_{}(G_1(\btheta),\bb) \end{pmatrix}$, yields
\begin{equation}\label{NL-GS}
    {\bf a}(\btheta,\btheta^*)= \Bigl( B_3(\btheta) - B(\btheta^{*})\Bigr)B^{-1}_1(\btheta) B_2(\btheta) \nabla_{\scriptsize\btheta}^2F(\btheta^*) (\btheta -\btheta^*) -B(\btheta^{*}) B^{-1}_1(\btheta)\hat{R}(\btheta;\btheta^*).  
\end{equation}

By the assumptions that $F(\btheta)$ is twice continuously differentiable and that $\nabla_{\scriptsize\btheta}^2F(\btheta^*)$ is SPD, there exists a $\delta > 0$ neighbourhood centered at $\btheta^{*}$, ${\cal B}(\btheta^{*}; \delta) = \{\btheta \in \mathbb{R}^{2n+1} : \|\btheta - \btheta^{*}\| < \delta\} \subset {\cal O}$, such that $\|B^{-1}(\btheta) \|$ and $\|B^{-1}_1(\btheta) \|$ are bounded for all $\btheta \in {\cal B}(\btheta^{*}; \delta)$ and that 
\[
\lim_{\bm{\scriptstyle\theta} \to \bm{\scriptstyle\theta}^{*}} \Bigl( B(\btheta) - B(\btheta^{*})\Bigr) = {\bf 0} \quad\mbox{and}\quad \lim_{\bm{\scriptstyle\theta} \to \bm{\scriptstyle\theta}^{*}} \Bigl( B_3(\btheta) - B(\btheta^{*})\Bigr) = {\bf 0}.
\]
Now, \cref{lim1} is a direct consequence of \cref{L-GS}, \cref{NL-GS}, the triangle inequality, and the facts that 
\[
\lim\limits_{{\scriptsize\btheta} \to {\scriptsize\btheta^{*}}} \|{R}(\btheta;\btheta^*)\| \big/ \|\btheta - \btheta^*\| = 0 \quad\mbox{and}\quad  \lim\limits_{{\scriptsize\btheta} \to {\scriptsize\btheta^{*}}} \|\hat{R}(\btheta;\btheta^*)\| \big/ \|\btheta - \btheta^*\| = 0.
\]
This completes the proof of the lemma.
\end{proof}

The following lemma was stated and proved in \cite{Rob}. For convenience of readers, a brief proof is provided. 

\begin{lemma}\label{lemmPDS}
Assume that $A \in \mathbb{R}^{m\times m}$ is SPD and that $M  \in \mathbb{R}^{m\times m}$ is invertible. Then the matrix $M + M^T - A$ is SPD if and only if
    \begin{equation*}
        \|I_{2n+1} - M^{-1}A \|_{A} < 1,
    \end{equation*}
where $\|\bv\|_A=\sqrt{\langle A\bv, \bv \rangle}$ and $\langle\cdot, \cdot\rangle$ is the standard inner product in $\mathbb{R}^{m}$.
\end{lemma}
\begin{proof}
For any $\bv \in \mathbb{R}^{m}$, it is easy to check that 
\[
 \|(I - M^{-1}A)\bv\|_A^2 = \langle A\bv, \bv \rangle - \left\langle(M + M^T -A)(M^{-1}A)\bv, (M^{-1}A)\bv \right\rangle ,
\]
which implies the validity of the lemma.
\end{proof}

Now, we are ready to present a sufficient condition for local convergence of the fixed-point iteration in \cref{fixedPoint} for the NL-GS and L-GS.

\begin{theorem}\label{main_lemma}
If $\nabla^2_{\scriptsize{\btheta}}F(\btheta^{*})$ is SPD, then  
the fixed point iteration \cref{fixedPoint} for the L-GS or NL-GS converges locally to $\btheta^{*}$ in the norm induced by $\nabla^2_{\scriptsize{\btheta}}F(\btheta^{*})$.
\end{theorem}
\begin{proof}
Symmetry and positive definiteness of the Hessian $\nabla^2_{\scriptsize{\btheta}}F(\btheta^{*})$ implies that
\[
B(\btheta^{*})^T + B(\btheta^{*}) - \nabla^2_{\scriptsize{\btheta}}F(\btheta^{*}) = \begin{pmatrix} {\cal H}_{11}(\btheta^{*}) & {\bf 0}\\{\bf 0} &{\cal H}_{22}(\btheta^{*}) \end{pmatrix}
\]
is SPD. Hence, by \cref{lemmPDS}, we have 
\[
\|I_{2n+1} - B(\btheta^*)^{-1}\nabla^2_{\scriptsize{\btheta}}F(\btheta^{*}))\|_{\nabla^2_{\text{\tiny{\btheta}}}F(\text{\scriptsize{\btheta}}^*)} = \sigma < 1,
\]
which, together with \cref{thmOS} and \cref{BN_lemma1}, implies the validity of the theorem.
\end{proof}

\begin{remark}
For the Jacobi method described in \cref{Jacobi}, if $\nabla^2 F(\btheta^*)$ is SPD, then a similar argument as that of  \cref{BN_lemma1} gives the following Jacobian matrix $\bJ_{G}(\btheta^{*})$ of $G(\btheta)$ at the fixed point $\btheta^*$
\[
        \bJ_{G}(\btheta^{*}) = I_{2n+1} - \tilde{B}(\btheta^{*})^{-1}\nabla^2_{\scriptsize{\btheta}}F(\btheta^{*}), \quad\mbox{where }\,  \tilde{B}(\btheta) = \begin{pmatrix} {\cal H}_{11}(\btheta) & {\bf 0}\\[1mm]{\bf 0} &{\cal H}_{22}(\btheta) \end{pmatrix}.
    \]
Moreover, the iteration is locally convergent to $\btheta^*$ if and only if
    \[
        \tilde{B}(\btheta^{*})^T + \tilde{B}(\btheta^{*}) - \nabla^2_{\scriptsize{\btheta}}F(\btheta^{*}) 
        = \begin{pmatrix} {\cal H}_{11}(\btheta^{*}) & -{\cal H}_{12}(\btheta^*)\\[1mm]-{\cal H}_{21}(\btheta^*) &{\cal H}_{22}(\btheta^{*}) \end{pmatrix}
    \]
    is SPD.
\end{remark}

\section{Applications}\label{Apps}

This section applies the local convergence theory developed in the previous section to one-dimensional least-squares (LS) function approximation and diffusion-reaction (DR) problems studied in \cite{cai2024fast2}. 

For the diffusion-reaction problem, let $u$ be the exact solution of the differential equation
\begin{equation}\label{pde}
    \left\{\begin{array}{lr}
        -(a(x)u^{\prime}(x))^{\prime} +r(x)u(x)=f(x), & \mbox{in }\, I=(0,1),\\[2mm]
        u(0)= \alpha,  \quad u(1)=\beta,&
    \end{array}\right.
\end{equation}
where the diffusion coefficient $a(x)$, the reaction coefficient $r(x)$, and the right-hand side $f(x)$ are given real-valued functions defined on $I$. Assume that $a(x)$ and $r(x)$ are bounded below by the respective positive constant $\mu > 0$ and non-negative constant $r_0\geq 0$ almost everywhere on $I$. 
Assume that $f(x) \in C(I)$ and that $a(x) \in C^1(I)$.
For the least-squares approximation problem, we assume that the weight function $r(x)\ge r_0>0$. For these problems, the NN approximation is to seek $u_n^*(x) = u_n(x; \btheta^*)\in {\cal M}_n(I)$ such that
\begin{equation}\label{min}
    u_n^*(x) = u_n(x;\btheta^*)=\argmin_{u_n\in\cM_n(I)}J\left(u_n(\cdot;\btheta)\right),
\end{equation} 
where the functional $J(v)$ is given by
\begin{equation}\label{functional}
J(v)=\left\{\begin{array}{ll}
   \dfrac{1}{2}\displaystyle\int_{0}^1\left[a(x)(v^{\prime}(x))^2 + r(x)(v(x))^2\right]dx - \displaystyle\int_{0}^1f(x)v(x)dx + \dfrac{\gamma}{2}(v(1) - \beta)^2,  & \mbox{DR},  \\[4mm]
   \dfrac{1}{2}
\displaystyle\int_0^1 r(x)\left[(v(x) - u(x)\right]^2dx,  &  \mbox{LS}.
\end{array}
\right.
\end{equation}

\subsection{Hessian}
First, we calculate the Hessian matrices. To this end, denote by 
\begin{equation*}
H(t)=
    \begin{cases}
        1, & \quad t>0,\\
        \frac{1}{2}, & \quad t = 0,\\
        0, & \quad t < 0,
    \end{cases}
    \quad \mbox{and} \quad \delta(t)= 
    \begin{cases}
        +\infty, & \quad t=0,\\
        0, & \quad t \not= 0
    \end{cases}
\end{equation*}
the Heaviside (unit) step function and the Dirac delta function, respectively. Clearly, $H(t)=\sigma'(t)$ and $\delta(t)=\sigma''(t)$ everywhere except at $t=0$. For each $i = 0, 1, \dots, n$, let 
\begin{equation*}
    \sigma_i(x)=\sigma(x-b_i), \quad  H_i(x)=H(x-b_i), \quad \mbox{and} \quad \delta_{i}(x)=\delta(x-b_i),
\end{equation*}
and let
\begin{equation*}
    {\bf \Sigma}_{n+1}(x)=\begin{pmatrix}
        \sigma_0(x)\\\sigma_1(x)\\
        \vdots\\
        \sigma_n(x)
    \end{pmatrix},\: \bH_n(x)=\begin{pmatrix}H_1(x)\\
    \vdots\\H_n(x)       
    \end{pmatrix},\: \bH_{n+1}(x) = \begin{pmatrix}
        H_{0}(x)\\
        H_1(x)\\
        \vdots\\
        H_{n}(x)
    \end{pmatrix},\,\,\mbox{and } {\bf \Lambda}_n(x)=\begin{pmatrix}
        \delta_{1}(x)\\
        \vdots \\
        \delta_n(x)
    \end{pmatrix}.
\end{equation*}

For $i = 1, \dots, n$, let  
\begin{equation}\label{gi}
    g_i= g_i(\btheta)= \left\{\begin{array}{ll}  r(b_i)u_n(b_i)-f(b_i) - a^\prime(b_i)u_n'(b_i) ,  & \mbox{DR},  \\[3mm]   r(b_i)(u_n(b_i) - u(b_i)), &  \mbox{LS},
\end{array}
\right.
\end{equation}
where $u_n'(b_i):=\sum\limits_{j=0}^{i-1}c_j + \dfrac{c_i}{2}$, and set
\begin{equation*}
    \bg(\btheta) = (g_1, \dots, g_n)^T, \quad \bd  = (1-b_1, \dots, 1-b_n)^T, \quad  \mbox{and} \quad \hat{\bc} = (c_1, \cdots, c_n)^T.
\end{equation*}
For any $u_n(x;\btheta)\in {\cal M}_n(I)$, the value of the functional $J(u_n(\cdot; \btheta))$ at $u_n(x;\btheta)$ is a function of parameters $\btheta$. For simplicity of notation, denote this function by $F$, i.e., $F(\btheta) = J(u_n(\cdot; \btheta))$.
In \cite{cai2024fast2}, we derived the principle blocks of the Hessian matrix $\nabla^2_{\scriptsize{\btheta}}F(\btheta)$ as follows
\begin{align}\label{H_11}
   & {\cal H}_{11}(\btheta) = \left\{\begin{array}{ll}
   \displaystyle\int_0^1a(x){\bf H}_{n+1}{\bf H}_{n+1}^Tdx + \displaystyle\int_0^1r(x){\bf \Sigma}_{n+1}{\bf \Sigma}_{n+1}^Tdx+\gamma\bd\bd^T,  & \mbox{DR},  \\[4mm]
\displaystyle\int_{0}^1r(x){\bf \Sigma}_{n+1}{\bf \Sigma}_{n+1}^Tdx, &  \mbox{LS}
\end{array}
\right. \\[4mm] \label{H_22}
\mbox{and}\quad & {\cal H}_{22}(\btheta) = \left\{\begin{array}{ll}
\bD(\hat{\bc})\bD(\bg)+  {\bf D}(\hat{\bc})\left(\displaystyle\int_0^1r(x){\bf H}_n{\bf H}^T_n dx\right) {\bf D}(\hat{\bc}) + \gamma \hat{\bc}\hat{\bc}^T,  & \mbox{DR},  \\[4mm]
\bD(\hat{\bc}) \bD(\bg)  + \bD(\hat{\bc}) \left(\displaystyle\int_{0}^1r(x)\bH_{n}\bH^T_n dx\right)\bD(\hat{\bc}), &  \mbox{LS},
\end{array}
\right.
\end{align}
where $\bD(\bg) =\text{diag}(g_1, \dots, g_n)$ and $\bD(\hat{\bc}) = \text{diag}(c_1, \dots, c_n)$ are diagonal matrices.

 To compute the off-diagonal block ${\cal H}_{12}(\btheta) = \left( \dfrac{\partial^2 F(\btheta)}{\partial c_i\partial b_j} \right)_{(n+1)\times n}$ of the Hessain matrix $\nabla^2_{\scriptsize{\btheta}}F(\btheta)$,  let 
\[
F_j(\btheta)= \left\{\begin{array}{ll}
\displaystyle\int_{0}^1  \biggl(f(x)H_{j}(x)-a(x)u_n'(x)\delta_j{(x)} -r(x)u_n(x)H_j(x) \biggr)dx  - \gamma( u_n(1) -\beta),  & \mbox{DR},  \\[4mm]
\displaystyle\int_{0}^1r(x)(u(x)-u_n(x))H_j(x)dx, &  \mbox{LS}.
\end{array}
\right.
\]
for $j = 1, \dots, n$. From (3.5) and (6.3) in \cite{cai2024fast}, we have
\begin{equation}\label{grad-bj}
    \frac{\partial}{\partial b_j} F(\btheta) = c_j F_j(\btheta).
\end{equation}
then together with the fact that $\dfrac{\partial c_j}{\partial c_i} = \delta_{ij}$ (the Kronecker delta), we obtain
\begin{equation}\label{H12-ij}
    \frac{\partial^2 F}{\partial c_i \partial b_j} (\btheta) 
    = \delta_{ij} F_j(\btheta) + c_j \frac{\partial F_j}{\partial c_i} (\btheta).
\end{equation}

\begin{lemma}\label{matrixH12}
         For the functional in \cref{functional}, the off-diagonal block ${\cal H}_{12}(\btheta)$ has the form 
\begin{equation*}\label{matrixQ_DR}
        {\cal H}_{12}(\btheta) = \left\{\begin{array}{ll}
{\cal R}(\btheta) -\left(\displaystyle\int_0^1a(x)\, \bH_{n+1} {\bf \Lambda}_n^Tdx + \displaystyle\int_{0}^1r(x)\,{\bf \Sigma}_{n+1}\bH_n^Tdx\right)\bD(\hat{\bc}) - \gamma\bd\hat{\bc}^T,  & \mbox{\textup{DR}},  \\[4mm]
{\cal R}(\btheta) - \left(\displaystyle\int_{0}^1r(x){\bf \Sigma}_{n+1} \bH_n^Tdx\right)\bD(\hat{\bc}), &  \mbox{\textup{LS}},
\end{array}
\right.
   \end{equation*}
where $F_0(\btheta) = 0$ and  ${\cal R}(\btheta)=\left( {\cal R}_{ij}(\btheta)\right)=\left( \delta_{ij} F_i(\btheta)\right)_{(n+1)\times n}$.
\end{lemma}
\begin{proof}
The lemma is a direct consequence of \cref{H12-ij} and the fact that 
\[
\dfrac{\partial F_j}{\partial c_i} (\btheta) = \left\{\begin{array}{ll}
  -\displaystyle\int_{0}^1a(x)H_i(x)\delta_j(x)dx - \displaystyle\int_0^1r(x)\sigma_i(x)H_j(x) dx - \gamma (1 - b_i),  & \mbox{DR},  \\[4mm]
-\displaystyle\int_0^1r(x)\sigma_i(x)H_j(x) dx, &  \mbox{LS},
\end{array}
\right.
\]
for all $i = 0, 1, \ldots, n$.

\end{proof}





For a given function $w: I \rightarrow  \mathbb{R}$, define the matrices
\[
\mathbb{H}_{\Sigma}(w;\btheta) = 
\begin{pmatrix}
        \displaystyle\int_{0}^1 w(x) {\bf \Sigma}_{n+1}{\bf \Sigma}^T_{n+1} dx & - \displaystyle\int_{0}^1 w(x) {\bf \Sigma}_{n+1}\bH_{n}^Tdx\\[6mm]
         - \displaystyle\int_{0}^1 w(x) \bH_{n}{\bf \Sigma}_{n+1}^Tdx &    \displaystyle\int_0^1 w(x) \bH_{n}\bH_n^T dx 
    \end{pmatrix}
\]
and
\[
\mathbb{H}_{\Lambda}(w;\btheta)=\begin{pmatrix}  \displaystyle\int_0^1 w(x){\bf H}_{n+1}{\bf H}_{n+1}^Tdx & - \displaystyle\int_{0}^1 w(x)\bH_{n+1}{\bf\Lambda}_n^{T} dx\\[6mm]
                 - \displaystyle\int_{0}^1 w(x){\bf \Lambda}_n\bH_{n+1}^T dx & {\bf 0} \end{pmatrix}.
\]
Let $\mathbb{H}_{\Sigma+\Lambda}(\btheta) = \mathbb{H}_{\Sigma}(r;\btheta) + \mathbb{H}_{ \Lambda}(a;\btheta)$. Combining \cref{H_11}, \cref{H_22}, and \cref{matrixH12}, the Hessian matrix at $\btheta$ has the form
\begin{equation}\label{Hessian-theta}
    \nabla^2_{\scriptsize{\btheta}}F(\btheta) = \mathbb{R}(\btheta) +
    \left\{\begin{array}{ll} 
    {\cal D}(\hat{\bc}) \mathbb{H}_{\Sigma+ \Lambda}(\btheta) {\cal D}(\hat{\bc}) + {\cal D}_0(\btheta) + 
         \gamma \begin{pmatrix} \bd \\ -\hat{\bc} \end{pmatrix} \left( \bd^T, -\hat{\bc}^T\right),   & \mbox{DR},  \\[4mm]
     {\cal D}(\hat{\bc})\mathbb{H}_{\Sigma}(r;\btheta) {\cal D}(\hat{\bc}) + {\cal D}_0(\btheta), &  \mbox{LS},
\end{array}\right.
\end{equation}
where ${\cal D}(\hat{\bc}) = \begin{pmatrix} I_{n+1} & {\bf 0}\\ {\bf 0} & \bD(\hat{\bc})\end{pmatrix}$, ${\cal D}_0(\btheta) = \begin{pmatrix} {\bf 0} & {\bf 0}\\ {\bf 0} & \bD(\hat{\bc})\bD(\bg)\end{pmatrix}$, and $\mathbb{R}(\btheta) = \begin{pmatrix}
    {\bf 0} & {\cal R}(\btheta) \\
    {\cal R}^T(\btheta) & {\bf 0}
\end{pmatrix}$. \\


\subsection{Local Convergence} 
In this section, we present sufficient conditions for the local convergence of the BN method for both the diffusion-reaction and the least-squares approximation problems. 

First, we establish lower bounds of the smallest eigenvalues of the matrices $\mathbb{H}_{\Sigma}(w;\btheta)$ and $\mathbb{H}_{\Lambda}(w;\btheta)$.
To this end, let 
\begin{equation*}
    h_i = b_{i+1} - b_{i}, \quad i = 0, \dots, n; \quad h_{\text{min}}=\displaystyle\min_{0\leq i\leq n}\{h_i\}, \quad \tilde{h}_i= \min\{h_{i-1}, h_{i}\},  \quad i= 1\, \dots, n.\end{equation*}
    
\begin{lemma}\label{QuadLS}
Suppose that $0\leq w_0\leq w(x) \in C(I)$. 
Then, for any $\balpha = (\alpha_0, \alpha_1, \dots, \alpha_n)^T \in \mathbb{R}^{n+1}$, $\bbeta = (\beta_1, \dots, \beta_n)^T \in \mathbb{R}^n$, and $\tau \in (0,1]$, we have
\begin{equation}\label{l-b}
    \left\{\begin{array}{l} 
     \quad\,\, \quad (\balpha^T, \bbeta^T) \mathbb{H}_{\Sigma}(w;\btheta)\begin{pmatrix} \balpha \\
                \bbeta \end{pmatrix} \geq \dfrac{w_0h_{\textup{min}}^3}{96} |\balpha|^2 + \dfrac{w_0}{24}\sum\limits_{i=1}^n\tilde{h}_i\beta_i^2  \\[6mm]
      \mbox{and}\quad (\balpha^T, \bbeta^T) \mathbb{H}_{\Lambda}(w;\btheta)\begin{pmatrix} \balpha \\
                \bbeta \end{pmatrix}    \geq \dfrac{w_0 (1-\tau)h_{\textup{min}}}{4} |\balpha|^2 - \dfrac{1}{2\tau w_0}\sum\limits_{i=1}^n  w(b_i)^2\left(h_{i-1}^{-1} + h_{i}^{-1}\right)\beta_i^2,
                \end{array}\right.
\end{equation}
where $|\balpha|$ is the magnitude of $\balpha$.
\end{lemma}

\begin{proof}
For any $\balpha = (\alpha_0, \alpha_1, \dots, \alpha_n)^T \in \mathbb{R}^{n+1}$ and any $\bbeta = (\beta_1, \dots, \beta_n)^T \in \mathbb{R}^n$, let $\beta_0 = 0$ and define the following piecewise linear and constant functions
\[
\alpha_i(x) = \sum_{j=0}^i \alpha_j \sigma_j(x)= \sum_{j=0}^i \alpha_j \sigma(x-b_j) \quad\mbox{and}\quad \beta_i(x) = \sum_{j=0}^i \beta_j H_j(x) \quad\mbox{for }\, i = 0, 1, \ldots, n,
\]
respectively. For $x\in (b_i,b_{i+1})$, $\alpha_n(x) = \alpha_i(x)$ and $\beta_n(x) = \beta_i(x)= \sum\limits_{j=0}^i \beta_j :=a_i$. Clearly, we have
\begin{equation}\label{quad-HS}
    \left(\balpha^T, \bbeta^T\right) \mathbb{H}_{\Sigma}(w;\btheta) 
\begin{pmatrix} \balpha \\ \bbeta \end{pmatrix} 
= \int_0^1 \!\! w \left(\alpha_n(x) - \beta_n(x)\right)^2 dx \ge w_0 \sum_{i=0}^n \int_{b_i}^{b_{i+1}} \!\! \left(\alpha_i(x) - a_i\right)^2 dx.
\end{equation}
By the Simpson rule, $\alpha_i(b_i)=\alpha_{i-1}(b_i)$, and $\sum\limits_{i=0}^n (\alpha_i(b_{i+1}) - a_i)^2\ge \sum\limits_{i=1}^n (\alpha_{i-1}(b_{i}) - a_{i-1})^2$, we obtain
\begin{align*}
& 12 \sum_{i=0}^n\int_{b_i}^{b_{i+1}} (\alpha_i(x) - a_i)^2 dx \ge 2 \sum_{i=0}^n h_i \left[(\alpha_i(b_i) - a_i)^2 + (\alpha_i(b_{i+1}) - a_i)^2  \right] \\[2mm]
\ge \, & \sum_{i=0}^n h_i \left[(\alpha_i(b_i) - a_i)^2 + (\alpha_i(b_{i+1}) - a_i)^2  \right] + \sum_{i=1}^n h_i (\alpha_{i-1}(b_i) - a_i)^2 + \sum_{i=1}^n h_{i-1} (\alpha_{i-1}(b_{i}) - a_{i-1})^2 \\[2mm]
\ge \, & \dfrac12 \sum_{i=0}^n h_i \left(\alpha_i(b_i) - \alpha_{i}(b_{i+1})\right)^2 + \dfrac12 \sum_{i=0}^n \tilde{h}_i \left(a_i - a_{i-1}\right)^2
= \dfrac12 \sum_{i=0}^n h^3_i \left(\sum_{j=0}^i\alpha_j\right)^2 + \dfrac12 \sum_{i=0}^n \tilde{h}_i \beta_i^2.
\end{align*}
Now, the first inequality in \cref{l-b} is a direct consequence of \cref{quad-HS} and the following inequality
\begin{equation}\label{commonInequality}
|\balpha|^2 = \sum_{i=0}^n \left(\sum_{j=0}^i \alpha_j - \sum_{j=0}^{i-1}\alpha_j\right)^2
\leq 2\sum_{i=0}^n \left(\sum_{j=0}^i \alpha_j\right)^2 
  + 2\sum_{i=0}^n \left(\sum_{j=0}^{i-1} \alpha_j\right)^2
\leq 4\sum_{i=0}^n \left(\sum_{j=0}^i \alpha_j\right)^2.
\end{equation}

To prove the validity of the second inequality in \cref{l-b}, let  $\delta^{\epsilon+}(t)$ and $\delta^{\epsilon-}(t)$ be approximations to the Dirac delta function such that 
\begin{equation*}
\delta^{\epsilon+}(t)=
    \begin{cases}
        \dfrac{1}{\epsilon + \rho}, & \quad t\in [ - \rho/2 ,\epsilon + \rho/2],\\
        
        0, & \quad\text{otherwise},
    \end{cases}
    \quad \mbox{and} \quad \delta^{\epsilon-}(t)=
    \begin{cases}
        \dfrac{1}{\epsilon + \rho}, & \quad t\in [ - \rho/2 -\epsilon, \rho/2],\\
        
        0, & \quad\text{otherwise},
    \end{cases}
\end{equation*}
 for any $\epsilon>0$ and $\rho \in (0, h_{\text{min}})$.

For each $i = 1, \dots, n$, let $\epsilon_i =  \tau h_i/2$,  $\delta_i^{+}(x) = \delta^{\epsilon_i+}(x-b_i)$, and $\delta_i^{-}(x) = \delta^{\epsilon_{i-1}-}(x-b_i)$.  Set
\begin{align*}
&\hat{\alpha}_n(x) = \sum_{i=0}^n \alpha_i H_i(x), \quad
\hat{\beta}_n(x) = \sum_{i=1}^n \beta_i \delta_i(x),\\[2mm] &\mbox{and} \quad 
\beta_{+}(x) = \sum_{i=1}^n w(b_i)\beta_i \delta_i^{+}(x), \quad \beta_{-}(x) = \sum_{i=1}^n w(b_i)\beta_i \delta_i^{-}(x) .
\end{align*}
For all $i = 0, \dots, n$ and $j = 1, \dots, n$, it is easy to check that
\[ 
\int_{0}^1 w(x) H_i(x) \delta_j(x)\,dx = \frac{1}{2}\left(\int_{0}^1 w(b_j) H_i(x) \delta_j^{+}(x)\,dx + \int_{0}^1 w(b_j) H_i(x) \delta_j^{-}(x)\,dx\right),
\] 
which, together with multiplying by $\alpha_i\beta_j$ and summing over $i$ and $j$, implies 
\[ 
\int_{0}^1 w(x)\hat{\alpha}_n(x)\hat{\beta}_n(x)\,dx = \frac{1}{2}\left(\int_{0}^1 \hat{\alpha}_n(x)\beta_{+}(x)\,dx + \int_{0}^1 \hat{\alpha}_n(x)\beta_{-}(x)\,dx\right).
\] 
Then the quadratic form of $\mathbb{H}_{\Lambda}(w;\btheta)$ is bounded below by
\begin{align}
&\notag
 \left(\balpha^T, \bbeta^T\right) \mathbb{H}_{\Lambda}(w;\btheta) 
\begin{pmatrix} \balpha \\ \bbeta \end{pmatrix}
= \int_{0}^1 w(x)\hat{\alpha}_n^2(x)\,dx 
 - 2\int_{0}^1 w(x)\hat{\alpha}_n(x)\hat{\beta}_n(x)\,dx \\[2mm]\notag
\geq\, & w_0 \!\int_{0}^1\! \hat{\alpha}_n^2(x)dx 
 \!-\!\int_{0}^1\! \hat{\alpha}_n(x)\beta_{+}(x)\,dx  -\!\int_{0}^1\! \hat{\alpha}_n(x)\beta_{-}(x)\,dx
\\[2mm]\label{H_ll}=\!& \,\frac{1}{2}\int_0^1\! \Bigl(\sqrt{w_0}\hat{\alpha}_n(x) \! -\! \dfrac{1}{\sqrt{w_0}}\beta_{+}(x)\Bigr)^2\!dx + \frac{1}{2}\int_{0}^1\Bigl(\sqrt{w_0}\hat{\alpha}_n(x) \! -\! \dfrac{1}{\sqrt{w_0}}\beta_{-}(x)\Bigr)^2 dx\!\\[2mm]\notag&-\!\frac{1}{2w_0}\left(\int_{0}^1 \!\beta_{+}^2(x) + \beta_{-}^2(x)\,dx\right).
\end{align}
To bound the first term below, let $\beta_{n+1} = c_{n+1} = \beta_{0}=0$ and $\epsilon_0 = \frac{h_0}{2}$. It follows from the facts that $\hat{\alpha}_n(x)$ and $\beta_+(x)$ are piecewise constants, $h_i - \epsilon_i =  \epsilon_i + (1-\tau)h_i$ and  \cref{commonInequality} that  
\begin{align}
&\notag
\int_0^1 \Bigl(\sqrt{w_0}\hat{\alpha}_n(x) \! -\! \dfrac{1}{\sqrt{w_0}}\beta_+(x)\Bigr)^2 \!dx 
\\[2mm]\notag\, &\geq \sum_{i=0}^n \!\left(\int_{b_i}^{b_i  +  \epsilon_i + \rho/2} + 
\int_{b_i+\epsilon_i + \rho/2 }^{b_{i+1}-\rho/2} \right)\! \Bigl(\sqrt{w_0}\hat{\alpha}_n(x) \! -\! \dfrac{1}{\sqrt{w_0}}\beta_+(x)\Bigr)^2 dx \\[2mm]
\notag
=\, & \sum_{i=0}^n \left(\epsilon_i+\frac{\rho}{2}\right)\left(\sqrt{w_0}\sum_{j=0}^i \alpha_j - \dfrac{w(b_i)\beta_i}{(\epsilon_i+\rho)\sqrt{w_0}} \right)^2 + \sum_{i=0}^n (h_i - \epsilon_i -\rho)\left(\sqrt{w_0}\sum_{j=0}^i \alpha_j\right)^2 \\[2mm]
\notag 
\geq \, &  \sum_{i=0}^n \dfrac{\epsilon_i}{2} \left(\dfrac{w(b_i)\beta_i}{(\epsilon_i+\rho)\sqrt{w_0}}\right)^2 
 + \sum_{i=0}^n \left[(1-\tau)h_i - \rho\right]\left(\sqrt{w_0}\sum_{j=0}^i \alpha_j\right)^2 
\\[2mm]\label{H_lll}
\geq \, &\frac{1}{2w_0}\sum_{i=1}^n\frac{ \epsilon_i w^2(b_i)\beta_i^2}{(\epsilon_i + \rho)^2} 
  + w_0\frac{(1-\tau)h_{\text{min}} - \rho}{4}|\balpha|^2.
\end{align}
Similarly, for the second term in \cref{H_ll}, following an argument analogous to the one used for inequality \cref{H_lll}, we obtain 
\begin{align}
    &\notag
\int_0^1 \Bigl(\sqrt{w_0}\hat{\alpha}_n(x) \! -\! \dfrac{1}{\sqrt{w_0}}\beta_-(x)\Bigr)^2 \!dx 
\\[2mm]\notag\, &\geq \sum_{i=1}^{n+1} \!\left(\int_{b_{i-1}+\rho/2}^{b_{i-1}  -  \epsilon_{i-1} - \rho/2} + 
\int_{b_i-\epsilon_{i-1} - \rho/2 }^{b_{i}} \right)\! \Bigl(\sqrt{w_0}\hat{\alpha}_n(x) \! -\! \dfrac{1}{\sqrt{w_0}}\beta_-(x)\Bigr)^2 dx \\[2mm]\label{H_llll}\, &\geq \, \frac{1}{2w_0}\sum_{i=1}^n\frac{ \epsilon_{i-1} w^2(b_i)\beta_{i}^2}{(\epsilon_{i-1} + \rho)^2} 
  + w_0\frac{(1-\tau)h_{\text{min}} - \rho}{4}|\balpha|^2.
\end{align}
Since $\displaystyle\int_{0}^1 \!\beta_+^2(x)dx= \sum\limits_{i=1}^n (\epsilon_i+ \rho)^{-1} w^2(b_i)\beta_i^2$ and $\displaystyle\int_{0}^1 \!\beta_-^2(x)dx= \sum\limits_{i=1}^n (\epsilon_{i-1}+ \rho)^{-1} w^2(b_i)\beta_i^2$, after letting $\rho \rightarrow0$, it follows from \cref{H_ll}, \cref{H_lll} and \cref{H_llll} that  
\begin{align*}\label{H1ll}
    \left(\balpha^T, \bbeta^T\right) \mathbb{H}_{\Lambda}(w;\btheta) 
\begin{pmatrix} \balpha \\ \bbeta \end{pmatrix} \geq\dfrac{w_0 (1-\tau)h_{\textup{min}}}{4} |\balpha|^2 -  \dfrac{1}{2\tau w_0}\sum\limits_{i=1}^n  w(b_i)^2\left(h_{i-1}^{-1} + h_{i}^{-1}\right)\beta_i^2, 
\end{align*}
and the lemma is proved.
\end{proof}

\begin{lemma}\label{lemma_spd}
    For all $i \in \{ 1, \dots, n\}$, assume that $c_i \neq 0$ and that
    \begin{equation}\label{convCondGeneral}
       F_i^2(\btheta) < \left\{\begin{array}{ll} 
    c_i^2\left(\dfrac{r_0h^3_{\min}}{96} + \dfrac{\mu(1 - \tau) h_{\min}}{4} \right)\left(\dfrac{g_i}{c_i} + \dfrac{r_0\tilde{h}_i}{24} - \dfrac{a^2(b_i)}{2\tau \mu }\left(h_{i-1}^{-1} + h_{i}^{-1}\right) \right) ,   & \mbox{\textup{DR}},  \\[3mm]
      c_i^2\left(\dfrac{r_0h^3_{\min}}{96}  \right)\left(\dfrac{g_i}{c_i} + \dfrac{r_0\tilde{h}_i}{24} \right)  , &  \mbox{\textup{LS}},
\end{array}\right.  
    \end{equation}
for some $0<\tau <1$, then $\nabla^2_{\scriptsize{\btheta}} F(\btheta) =\nabla^2_{\scriptsize{\btheta}}J(u_n(\cdot;\btheta))$ is SPD.
\end{lemma}
\begin{proof}For any $\balpha = (\alpha_0, \alpha_1, \dots, \alpha_n)^T \in \mathbb{R}^{n+1}$ and $\bbeta = (\beta_1, \dots, \beta_n)^T \in \mathbb{R}^n$, let $\hat{\bbeta} = \bD(\hat{\bc}) \bbeta$. 
For the diffusion-reaction problem, it follows from \cref{Hessian-theta}, the fact that $\left(\balpha^T\bd - \bbeta^T\hat{\bc} \right)^2\ge 0$, \cref{l-b}, and the assumptions on the lower bounds of $a(x)$ and $r(x)$ that 
\begin{align*}
& (\balpha^T, \bbeta^T) \nabla^2_{\scriptsize{\btheta}}F(\btheta) \begin{pmatrix} \balpha \\ \bbeta \end{pmatrix}  \\[2mm]=\, &\,(\balpha^T, {\bbeta}^T) \mathbb{R}(\btheta)\begin{pmatrix} \balpha \\ \bbeta \end{pmatrix}+ (\balpha^T, \hat{\bbeta}^T) \mathbb{H}_{\Sigma+\Lambda}(\btheta) \begin{pmatrix} \balpha \\ \hat{\bbeta} \end{pmatrix} + \hat{\bbeta}^T \bD(\bg) \bD^{-1}(\hat{\bc}) \hat{\bbeta} + \gamma\, \left(\balpha^T\bd - \bbeta^T\hat{\bc} \right)^2 \\[2mm]
 \ge \, &\, (\balpha^T, {\bbeta}^T) \mathbb{R}(\btheta)\begin{pmatrix} \balpha \\ \bbeta \end{pmatrix} + (\balpha^T, \hat{\bbeta}^T) \mathbb{H}_{\Sigma+\Lambda}(\btheta) \begin{pmatrix} \balpha \\ \hat{\bbeta} \end{pmatrix} + \hat{\bbeta}^T \bD(\bg) \bD^{-1}(\hat{\bc}) \hat{\bbeta} \\[2mm]
  \ge \, &\, 2\sum_{i= 1}^n\alpha_i\beta_iF_i+  \dfrac{r_0(h_{\textup{min}})^3}{96} |\balpha|^2 + \dfrac{r_0}{24}\sum\limits_{i=1}^n\tilde{h}_i(c_i\beta_i)^2 \\[2mm]&+ \dfrac{\mu(1-\tau) h_{\textup{min}}}{4} |\balpha|^2 + \sum\limits_{i=1}^n \left(\dfrac{g_i}{c_i} -\dfrac{a^2(b_i)}{2\mu}(h_{i-1}^{-1} + h_{i}^{-1})\right)   (c_i\beta_i)^2  \\[2mm]
  = \, & \, \sum_{i =1}^n\Biggl[ \left(\dfrac{r_0h^3_{\min}}{96} + \dfrac{\mu(1-\tau) h_{\min}}{4} \right)\alpha_i^2 + 2\alpha_i\beta_iF_i + \left(\dfrac{g_i}{c_i} + \dfrac{r_0\tilde{h}_i}{24} - \dfrac{a^2(b_i)}{2\tau \mu }(h_{i-1}^{-1} + h_{i}^{-1})\right)c_i^2\beta_i^2 \Biggr].
\end{align*}   
In a similar fashion, we have
\[
(\balpha^T, \bbeta^T) \nabla^2_{\scriptsize{\btheta}}F(\btheta) \begin{pmatrix} \balpha \\ \bbeta \end{pmatrix} \ge  \sum_{i =1}^n\Biggl[ \left(\dfrac{r_0h^3_{\min}}{96} \right)\alpha_i^2 + 2\alpha_i\beta_iF_i + \left(\dfrac{g_i}{c_i} + \dfrac{r_0\tilde{h}_i}{24} \right)c_i^2\beta_i^2 \Biggr]
\]
for the least-squares approximation problem. Notice that $a_1x^2 + 2a_2xy + a_3y^2 > 0$ for all $(x,y)\not= (0,0)$ if $a_1 > 0$, $a_3 > 0$, and $a_2^2 < a_1a_3$. Now, the positive definiteness of $\nabla^2_{\scriptsize{\btheta}} F(\btheta)$ is a direct consequence of \cref{convCondGeneral}.
\end{proof}

If $c_j\not=0$, by \cref{grad-bj}, $F_j(\btheta)$ vanishes at a critical point $\btheta^*$. Hence, ${\cal R}(\btheta^*)={\bf 0}$, which gives
\begin{equation*}\label{Hessian}
    \nabla^2_{\scriptsize{\btheta}}F(\btheta^{*}) = 
    \left\{\begin{array}{ll} 
    {\cal D}(\hat{\bc}^*) \mathbb{H}_{\Sigma+ \Lambda}(\btheta^*) {\cal D}(\hat{\bc}^*) + {\cal D}_0(\btheta^*) + 
         \gamma \begin{pmatrix} \bd \\ -\hat{\bc}^* \end{pmatrix} \left( \bd^T, -(\hat{\bc}^*)^T\right),   & \mbox{DR},  \\[3mm]
     {\cal D}(\hat{\bc}^*)\mathbb{H}_{\Sigma}(r;\btheta^*) {\cal D}(\hat{\bc}^*) + {\cal D}_0(\btheta^*) , &  \mbox{LS}.
\end{array}\right.
\end{equation*}

\begin{theorem}\label{thmDR} For an open set ${\cal O} \subseteq \mathbb{R}^{2n+1}$ satisfying the invertibility assumption, let $\btheta^* = \begin{pmatrix} \bc^{*} \\ \bb^{*} \end{pmatrix} \in {\cal O}$ be a minimizer of problem \cref{min}. For all $i\in \{1, \ldots, n\}$, assume that $c^*_i \neq 0$ and that
\begin{equation}\label{convCond2}
\dfrac{g^*_i}{c^*_i}+\dfrac{r_0\tilde{h}^*_i}{24} > \left\{\begin{array}{ll}
  \dfrac{a^2(b^*_i)}{2\mu}\left(\dfrac{1}{h_{i-1}^*} + \dfrac{1}{h_{i}^{*}}\right),   &  \mbox{\textup{DR}},  \\[4mm]
   0,  & \mbox{\textup{LS}}.
\end{array}\right.
\end{equation}
Then the Hessian matrix at the critical point $\btheta^{*}$, $\nabla^2_{\scriptsize{\btheta}} F(\btheta^{*}) =\nabla^2_{\scriptsize{\btheta}}J(u_n(\cdot;\btheta^{*}))$, is SPD. Moreover, the BN method using either NL-GS or L-GS converges locally to $\btheta^{*}$ in the norm induced by $\nabla^2_{\scriptsize{\btheta}}F(\btheta^{*})$.
\end{theorem}

\begin{proof}This is a direct consequence of \cref{lemma_spd}, \cref{convCond2}, and the fact that $F_i(\btheta^*) = 0$ for all $i = 1, \dots, n$. 
\end{proof}

\subsection{Feasibility of BN}\label{S_Condition}

This section discusses feasibility of the BN with either NL-GS or L-GS,  or equivalently, invertibility of ${\cal H}_{11}(\btheta)$ and ${\cal H}_{22}(\btheta)$. Invertibility of ${\cal H}_{11}(\btheta)$ was shown in \cite{cai2024fast2}; moreover, ${\cal H}_{11}(\btheta)$ is SPD.

Regarding ${\cal H}_{22}(\btheta)$, the formula in \cref{H_22} uses derivative of the diffusion coefficient $a(x)$. As discussed in Section~5 of \cite{cai2024fast}, if $a(x)$ is not differential at $b_i$ for some $i\in \{1,\ldots,n\}$, then $b_i$ lies at the physical interface. This means that the (mesh) breakpoint $b_i$ is already at a good location and hence should be fixed without further update. 


Below, we discuss invertibility of ${\cal H}_{22}(\btheta)$ at the $k^{\text{th}}$ iteration. For simplicity of notation, the parameters are still denoted by $\hat{\bc}$ and $\bb$ without the superscript $(k)$. The assumption $c_i \not= 0$ for all $i \in \{1, \ldots, n\}$ is a necessary condition for invertibility of ${\cal H}_{22}(\btheta)$. During the iterative process, this condition is not always satisfied. In the case that $c_l = 0$ for some $l \in \{1, \ldots, n\}$, then ${\cal H}_{22}(\btheta)$ is singular. To deal with such singularity, similar to \cite{Cai24GN, cai2024fast, cai2024fast2}, the nonlinear parameter $b_l$ is not updated at the current step, because the corresponding neuron has no contribution to the current approximation. When this happens several times to a neuron, then we can either remove or redistribute this neuron. Based on the above discussion, we introduce the following set of indices for the non-contributing neurons
\begin{equation}\label{setS1}
    S_1 = \bigl\{i \in \{1, \dots,n\}:|c_i|<\tau_1 \quad \mbox{or}\quad b_i\notin I \bigr\},
\end{equation}
where $0\leq \tau_1<1$ is a small parameter.


Next, at each step, to update the nonlinear parameters of neurons whose indices are not in $S_1$, 
let us denote by $\hat{\cal H}_{22}(\btheta)$ the reduced Hessian after removing neurons with indices in $S_1$. For the diffusion problem, $\hat{\cal H}_{22}(\btheta)$ is invertible if and only if $g_i\not= 0$ for all $i\notin S_1$. Under the assumption that $u_n(x) \approx u(x)$, i.e., $u_n(x)$ is a good approximation to $u(x)$, it was shown in \cite{cai2024fast} that 
\begin{equation}\label{gid}
    g_i\approx a(b_i)u^{\prime\prime}(b_i),
\end{equation}
where $g_i$ is defined in \cref{gi}. As discussed in \cite{cai2024fast}, $g_i\approx 0$ implies that $b_i$ is either at a nearly optimal location or no need for update. In both cases, $b_i$ is not updated. Hence, we introduce the following set of indices for non-updating neurons
\begin{equation}\label{setS2-DR}
 S_2 =\left\{i\in\{1,\dots,n\}\backslash S_1: \dfrac{|g_i|}{a(b_i)}  \leq  \tau_{2} \quad\mbox{or}\quad a'(b_i)\quad\mbox{DNE}  \right\},
\end{equation}
where $0\leq \tau_2<1$ is a small parameter.

The above strategy does apply to the diffusion-reaction problem because \cref{gi} and \cref{pde} imply the validity of \cref{gid}:
\[ g_i = r(b_i)u_n(b_i) - f(b_i) - a^\prime(b_i)u_n'(b_i)
    \approx r(b_i)u(b_i) - f(b_i) - a'(b_i)u'(b_i)
    = a(b_i)u''(b_i).
\]
Nevertheless, the sufficient condition on the positive definiteness of ${\cal H}_{22}(\btheta)$ was given in Lemma~5.2 of \cite{cai2024fast2}:
\begin{equation}\label{spdH22}
\frac{g_i}{c_i}+\frac{r_0\tilde{h}_i}{4}>0 \quad\mbox{for all}\quad i\in \{1,\dots,n\} \setminus S_1,
\end{equation}
where $r_0\ge 0$ for DR and $r_0>0$ for LS.
Notice that the linear parameter $c_i=u_n'({b_i}^+) - u_n'({b_i}^-)$ means the change of the slope of $u_n(x)$ at $b_i$ (see section~5 of \cite{cai2024fast}). Hence, $c_i\approx u'({b_i}^+) - u'({b_i}^-)$ and $u''(b_i)$ have the same sign, which implies  
\begin{equation*}\label{gi/ci}
\dfrac{g_i}{c_i} \approx \dfrac{a(b_i) u^{\prime\prime}(b_i)}{c_i} = \dfrac{a(b_i) |u^{\prime\prime}(b_i)|}{|c_i|},    
\end{equation*}
which, together with positivity of $a(b_i)$, implies \cref{spdH22} is almost always valid.



For the LS problem, \cref{spdH22} is sufficient but not necessary for the positive definiteness of ${\cal H}_{22}(\btheta)$. Moreover, positive definiteness differs from invertibility. Since $g_i=r(b_i) \left(u_n(b_i) - u(b_i)\right)$ is assumed to be small, for implementation purposes, we do not update the $i^{th}$ neuron if $g_i/c_i < - 1/\tau_3$ is negatively large, where parameter $\tau_3 \in (0,1)$ . In this case, $|c_i| < \tau_3|g_i|$ is small if $g_i$ is small. Hence, let
\begin{equation}\label{setS2-LS}
 S_2 =\left\{i\in\{1,\dots,n\}\backslash S_1: \frac{g_i}{c_i} < 0 \,\mbox{ and }\, |c_i| < \tau_3|g_i| \right\}.
\end{equation}
Notice that the reduced Hessian $\hat{\cal H}_{22}(\btheta)$ is not guaranteed to be nonsingular after removing neurons with indices in $S_1\cup S_2$. In cases where the reduced Hessian $\hat{\cal H}_{22}(\btheta)$ is not invertible, one might consider the Gauss-Newton matrix
\[{\cal G}_{22}(\btheta) = {\cal D}(\hat{\bc})\mathbb{H}_{\Sigma}(r;\btheta) {\cal D}(\hat{\bc}),\]
which is SPD under the assumption that $c_i \neq 0$ for all $i \in \{1, \dots, n\}$.

\subsection{Implementation}\label{implementatioSection} 

In this section, we outline the implementation of the BN method using a reduced nonlinear system constructed based on the sets \cref{setS1}, \cref{setS2-DR}, and \cref{setS2-LS}. The NL-GS scheme, implemented in earlier works \cite{cai2024fast,cai2024fast2}, is used to illustrate the implementation, which is also applicable to the other two schemes described in \cref{Section_BN_method}.

To this end, let 
\begin{equation*}
    S = \{1, \dots, n\} \backslash (S_1 \cup S_2)
\end{equation*}
be the set of indices for the neurons which remain in the system, and denote by $S^c=S_1\cup S_2$ the complement of $S$, the set of indices to be removed. Next, for a given vector $\bv \in \mathbb{R}^n$, denote by $\bv_{S} \in \mathbb{R}^{n-\left\lvert S^c\right\rvert}$ as the vector obtained by removing entries whose indices belong to $S^c$, where $\left\lvert S^c\right\rvert$ denotes the number of indices in $S^c$.
Similarly, for a matrix $\bB \in \mathbb{R}^{n\times n}$, define $\bB_{S} \in \mathbb{R}^{(n-\lvert S^c\rvert)\times (n-\lvert S^c\rvert)}$ as the matrix obtained by removing the rows of columns of $\bB$ corresponding to the indices in $S^c$. 
Then the \textit{reduced} search direction vector for the Newton step is defined as
\begin{equation}\label{reduce_Newt}
     \bd_{R}(\bc, \bb)= \bd_{R}(\btheta) = -\biggl({\cal  H}_{22}(\btheta)_{S}\biggr)^{-1}\bigg(\nabla_{\bb}F(\btheta)\biggr)_{S}.
\end{equation}

 We are now ready to present the reduced block Newton (rBN) algorithm with the NL-GS scheme (see \cref{alg:dBN} for pseudocode).
 Given the previous iterate $\bb^{(k)}$. Then the current iterate $\btheta^{(k+1)} =\begin{pmatrix}  \bc^{(k+1)} \\ \bb^{(k+1)} \end{pmatrix}$ is computed as follows:
\begin{itemize}
    \item[(i)] \textit{Compute the linear parameters} 
    \begin{equation*}
    \bc^{(k+1)} = \bc^{(k)} - {\cal H}_{11}(\bc^{(k)}, \bb^{(k)})^{-1}\nabla_{\bc}F(\bc^{(k)}, \bb^{(k)})
    \end{equation*}
    \item[(ii)] \textit{Compute the search direction $\mathbf{p}^{(k)}=\left(p_1^{(k)},\dots,p_n^{(k)}\right)^T$} by
    \begin{equation}\label{directionVector}
    \bigl(p^{(k)}_i \bigr)_{i \in S} = {\bd}_R(\bc^{(k+1)}, \bb^{(k)}) \quad \mbox{and} \quad \bigl(p^{(k)}_i\bigr)_{i \notin S} = {\bf 0},
    \end{equation}
    where $\bd_R(\bc^{(k+1)}, \bb^{(k)})$ is the \textit{reduced} Newton's direction vector defined in \cref{reduce_Newt}.
    \item [(iii)] \textit{Compute the nonlinear parameters}
    \begin{equation*}
        \bb^{(k+1)} = \bb^{(k)} +\mathbf{p}^{(k)}.
    \end{equation*}
    \item [(iv)] \textit{Redistribute non-contributing breakpoints} $b_i^{(k+1)}$ for all $i \in S_1$ and sort $\bb^{(k+1)}$.   
\end{itemize}

\begin{algorithm}
    \caption{Reduced block Newton (rBN) method for  \cref{min}} \label{alg:dBN} 
    \begin{algorithmic}
    \REQUIRE{Initial network parameters $\bb^{(0)}$}
    \ENSURE{Network parameters $\bc$, $\bb$}
    \FOR{$k=0,1\ldots$}
    \STATE $\triangleright$ \textit{Linear parameters}
    \STATE{$\bc^{(k+1)}\leftarrow \bc^{(k)} - {\cal H}_{11}(\bc^{(k)}, \bb^{(k)})^{-1}\nabla_{\bc}F(\bc^{(k)}, \bb^{(k)})$}
    \STATE $\triangleright$
    \textit{nonlinear parameters}
    \STATE{{\textit{Compute the search direction} }$\mathbf{p}^{(k)}$ \textit{as in}} \cref{directionVector}
    \STATE{$\bb^{(k+1)} \leftarrow \bb^{(k)} + \mathbf{p}^{(k)}$}
    \STATE $\triangleright$ \textit{Redistribute non-contributing neurons and sort $\bb^{(k+1)}$}
\ENDFOR
    \end{algorithmic}
\end{algorithm}

\begin{remark}\label{rmk_redist}
    The redistribution implemented for the numerical results shown in {\em \cite{cai2024fast, cai2024fast2}} in step {\em (iv)} was carried out as follows: For a neuron $b_l^{(k+1)}$ satisfying $l \in S_1$, we set 
    \[
    b_l^{(k+1)} \leftarrow\frac{b^{(k+1)}_{m-1}+b^{(k+1)}_{m}}{2},
    \]
    where $m\in\{1,\dots,n+1\}$ is an integer chosen uniformly at random.
\end{remark}

\subsubsection{Numerical 
Experiment}
To illustrate the performance of the reduced BN method outlined above, consider a singularly perturbed reaction-diffusion equation:
\begin{equation}\label{pde2}
    \left\{\begin{array}{lr}
        -\varepsilon^2u^{\prime\prime}(x) +u(x)=f(x), & x\in (-1,1),\\[2mm]
        u(-1) = u(1) = 0.
    \end{array}\right.
\end{equation}
 For $f(x) = -2\left(\varepsilon -4x^2\tanh{\left(\frac{1}{\varepsilon}(x^2 - \frac{1}{4})\right)}\right)\left(1/\cosh{\left( \frac{1}{\varepsilon}(x^2 - \frac{1}{4})\right)}\right)^2 + \tanh{\left(\frac{1}{\varepsilon}(x^2 - \frac{1}{4})\right)} - \tanh{\left(\frac{3}{4\varepsilon}\right)}$, problem \cref{pde2} has the following exact solution:
\begin{equation}\label{ex2pde}
    u(x) = \tanh{\left(\frac{1}{\varepsilon}(x^2 - \frac{1}{4})\right)} - \tanh{\left(\frac{3}{4\varepsilon}\right)}.
\end{equation}
Denote by $|\cdot|_1$ the $H^1$ seminorm on $(-1, 1)$. For $\nu = \varepsilon^2 = 10^{-6}$, the solution of \cref{ex2pde} exhibits sharp interior layers. It is well-known that overshooting and oscillations occur when using continuous piecewise linear approximation on a uniform partition (see \cref{fig2a} for $n = 16$). Using these uniform mesh points as an initial for the nonlinear parameters, 
$100$ iterations of the BN method moves them efficiently toward the interior layers, and the resulting approximation is greatly improved (see \cref{fig2b}). This example demonstrates importance of non-uniform mesh in approximation and efficiency of the BN method in non-convex optimization. 
For more numerical experiments, we refer readers to \cite{cai2024fast} and \cite{cai2024fast2}.

\begin{figure}[htbp]
    \centering
    \begin{subfigure}[b]{0.45\textwidth}
        \centering
        \includegraphics[width=\textwidth]{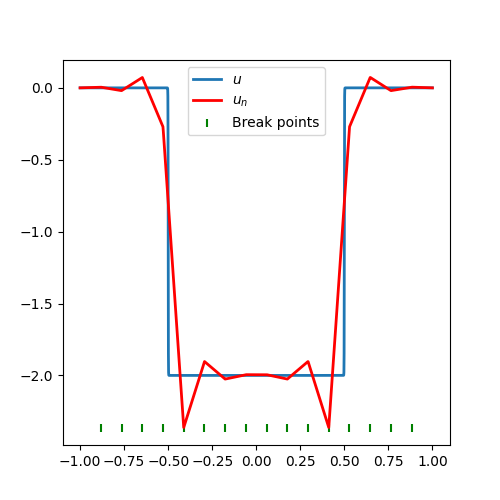}
        \caption{}\label{fig2a}
    \end{subfigure}
    \hfill
    \begin{subfigure}[b]{0.45\textwidth}
        \centering
        \includegraphics[width=\textwidth]{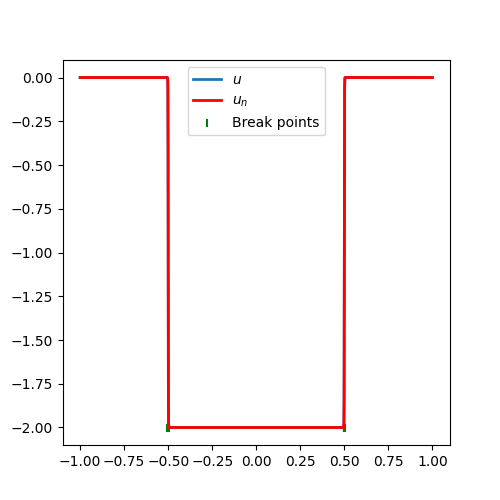}
        \caption{}\label{fig2b}
    \end{subfigure}
    \caption{For $\nu = \varepsilon^2 = 10^{-6}$: (a) initial NN model with 16 uniform breakpoints; $\frac{|u-u_n|_{1}}{|u|_1} = 0.988$, (b) optimized NN model with 16 breakpoints, 100 iterations, $\frac{|u-u_n|_{1}}{|u|_1} =0.173$.}\label{example2DR}
\end{figure}

\subsection{Reduced nonlinear System}\label{sec:rBN} This section studies the rBN method introduced in \cref{implementatioSection}. 
As explained in \cref{implementatioSection}, the iterative process may fix the location of certain breakpoints. We analyze the local convergence of the method under the assumption that no additional breakpoints will be fixed after a finite number of iterations. 

To this end, for a given $\alpha \in \mathbb{R}$, assume that there are $k < n$ fixed breakpoints denoted by
\[
0 =\tilde{b}_0< \tilde{b}_1 < \cdots < \tilde{b}_k < 1.
\]
Define the set of shallow ReLU neural networks with $k$ fixed and $n-k$ moving breakpoints as follows:
\begin{equation*}\label{fixed_NN}
    \mathcal{M}_{n,k}(I) =
    \left\{
        \alpha +
        \sum_{i = 0}^k c_i \sigma(x - \tilde{b}_i) +
        \sum_{j = 1}^{n-k} c_{k+j} \sigma(x - b_j)
        : c_{l} \in \mathbb{R},
        \; 0 < b_{1} < \cdots < b_{n-k} < 1
    \right\}.
\end{equation*}
For any $u_{n,k}(x) \in \mathcal{M}_{n,k}(I)$, denote the corresponding parameters by
\[
    \btheta_R =
    \begin{pmatrix}
        \bc \\[1mm]
        \bb_R
    \end{pmatrix}
    \in \mathbb{R}^{2n+1-k}
\]
where $\bc = (c_0, \dots, c_n)^T \in \mathbb{R}^{n+1}$ are again the linear parameters and $\bb_R = (b_1, \dots, b_{n-k})^T \in \mathbb{R}^{n-k}$ are the nonlinear parameters. After a certain number of the BN iterations, the optimization becomes to 
seek $u_{n,k}^*(x) = u_{n,k}(x; \btheta_R^*) \in \mathcal{M}_{n,k}(I)$ such that
\begin{equation*}
    u_{n,k}^*(x) = u_{n,k}^*(x; \btheta^*_R) = \argmin_{u_{n,k} \in \mathcal{M}_{n,k}(I)}
    J\!\left(u_{n,k}(\cdot; \btheta_R)\right).
\end{equation*}

The analysis presented in the preceding sections extends to the rBN method. 
Furthermore, following a similar reasoning as in \cref{thmDR}, if $c^*_{k+i} \neq 0$ and
\begin{equation*}
\dfrac{g^*_i}{c^*_{i+k}}+\dfrac{r_0\tilde{h}^*_i}{24} > \left\{\begin{array}{ll}
  \dfrac{a^2(b^*_i)}{2\mu}\left(\dfrac{1}{h_{i-1}^*} + \dfrac{1}{h_{i}^{*}}\right),    &  \mbox{DR},  \\[4mm]
   0,  & \mbox{LS}.
\end{array}\right.
\end{equation*}
for all $i \in \{1, \dots, n-k\}$, then the BN method applied to this reduced parameter set converges locally to $\btheta_R^* = \begin{pmatrix} \bc^{*} \\ \bb_R^{*} \end{pmatrix}$.

\begin{remark}
In \cref{Section_convergence}, we presented the local convergence analysis assuming the dimensions of the linear and nonlinear parameters are $n+1$ and $n$, respectively. For the rBN method discussed here, the dimensions are $n+1$ and $n-k$. However, the proofs of \cref{main_lemma} and the other lemmas in that section do not depend on the specific dimension of the parameter space. Therefore, the local convergence analysis extends directly to the rBN method. 
\end{remark}

\section{Conclusion}\label{sec:conc} A local convergence analysis for the BN methods including the reduced BN (rBN) method is presented in this paper. 
The rBN differs from the common optimization methods in reducing the number of parameters during the optimization process. We established a local convergence of the BN methods under some reasonable contiditions for one-dimensional diffusion–reaction problems and least-squares function approximation. The analysis may be extended to higher dimensions for shallow neural network approximation problems.




\bibliographystyle{plain}
\bibliography{ref}

\end{document}